\documentclass[11pt]{article}
\usepackage[colorlinks]{hyperref}
\usepackage{color}
\usepackage{graphicx}
\usepackage{graphics}
\usepackage{makeidx}
\usepackage{showidx}
\usepackage{latexsym}
\usepackage{amssymb}
\usepackage{verbatim}
\usepackage{amsmath}
\usepackage{amsthm}
\usepackage{amsfonts}
\usepackage{subcaption} 
\usepackage[abs]{overpic}
\usepackage{xcolor,varwidth}

\usepackage{geometry}
\newtheorem{theorem}{Theorem}[section]
\newtheorem{proposition}[theorem]{Proposition}
\newtheorem{remark}[theorem]{Remark}

\newtheorem{definition}[theorem]{Definition}
\newtheorem{corollary}[theorem]{Corollary}

\newtheorem{Conjecture}[theorem]{Conjecture}
\newtheorem{question}[theorem]{Question}

\newtheorem{example}[theorem]{Example}
\title{Dynamics and Topology of Conformally Anosov Contact 3-Manifolds}
\author{Surena Hozoori}
\begin{document}
\maketitle
\noindent
\begin{abstract}
We provide obstructions to the existence of conformally Anosov Reeb flows on a 3-manifold that partially generalize similar obstructions to Anosov Reeb flows. In particular, we show $\mathbb{S}^3$ does not admit conformally Anosov Reeb flows. We also give a Riemannian geometric condition on a metric compatible with a contact structure implying that a Reeb field is Anosov. From this we can give curvature conditions on a metric compatible with a contact structure that implies universal tightness of the contact structure among other things.
\end{abstract}
\vspace{.5cm}
\section{Introduction}
 Since early 1990s, it is known that Reeb dynamics can be exploited to extract contact topological information regarding the underlying contact 3-manifold. One of the first examples of such relation was discovered in Hofer's proof of Weinstein conjecture for overtwisted contact manifolds \cite{hofermain},
 
 \begin{theorem}[Hofer 1993]
 Any contact 3-manifold $(M,\xi)$ which admits a Reeb vector field without a contractible periodic Reeb orbit is tight.
 \end{theorem}
 
Similar ideas were used to establish subtle relations between dynamics and topology \cite{hofermain}\cite{hwz}\cite{hwz2}, followed by developments in early 2000s which resulted in the discovery of contact homology, a new Floer theoretic 
invariant of contact manifolds, generated by periodic orbits of Reeb vector fields \cite{sft}.

Therefore, it is natural to ask about the contact topological consequences of putting dynamical restrictions on Reeb vector fields associated to a contact manifold. In particular, we are interested in Reeb vector fields of certain {\em {\ttfamily"}Anosovity class{\ttfamily"}}.

It is well known that admitting an Ansov flow puts rigid restrictions on the topology of the underlying 3-manifold. For instance, it implies that the fundamental group of such a manifold needs to have exponential growth, which implies the non-existence of such flows on many manifolds, including $\mathbb{S}^3$ and $\mathbb{T}^3$ \cite{anosov}\cite{plante}. In the mid 1990s, Mitsumatsu \cite{Mitsumatsu} and Thurston-Eliashberg \cite{confoliations} introduced a more general class of flows which seemed to be more natural from topological and contact geometric viewpoint, namely {\em conformally  Anosov flows} (Mitsumatsu called these projectively Anosov flows). Although it was not clear to what extent this class of flows is bigger than Anosov flows, we know by now that conformally Anosov flows are abundant. In particular, we have infinitely many non-isotopic conformally Anosov flows on both $\mathbb{S}^3$ and $\mathbb{T}^3$ \cite{asa}. It was previously known that regularity assumptions on the stable and unstable directions of conformally Anosov flows makes them almost as rigid as Anosov flows \cite{asa0}\cite{noda}. However, it is hard to control the regularity of those line fields in a general setting.

A particularly interesting question in the context of contact topology is:

\begin{question}
What are the contact topological consequences of a contact manifold admitting an (conformally) Anosov Reeb vector field?
\end{question}

Similar problem can be motivated from the perspective of dynamical system. Recall that for a (conformally) Anosov flow, there exists a unique invariant transverse plane field formed by stable and unstable directions. Then the above question can be formulated as:

\begin{question}
What are the topological consequences of stable and unstable directions of an (conformally) Anosov flow forming a contact structure?
\end{question}

It turns out that admitting an Anosov Reeb vector field forces a lot of contact topological restrictions, like tightness, thanks to the fact that such flows do not contain any contractible periodic orbit. In this paper, we will show that relaxing this condition to admitting a conformally Anosov Reeb vector field does not take away many of the contact topological consequences, even if contractible Reeb orbits exist. This contrasts the phenomena in the general case of comparing Anosov and conformally Anosov flows. The proof relies on computing certain indices, namely \textit{Conley-Zehnder indices}, associated to the closed orbits of Reeb vector fields. Calling such contact structures (conformally) Anosov, our main theorem is

\begin{theorem}
Let $(M,\xi)$ be a conformally Anosov contact 3-manifold. Then $(M,\xi)$ is universally tight, irreducible and it does not admit any exact cobordism to $(\mathbb{S}^3 , \xi_{std})$.
\end{theorem}

In this paper, we assume $M$ to be a closed, connected and oriented 3-manifold.

\begin{remark}
The same phenomena seem to hold in higher dimensions as well, which is the subject of an upcoming paper. 
\end{remark}

It can be seen that, thanks to the classification of tight contact structures on $\mathbb{S}^3$ \cite{elover}\cite{elmart}, this implies non-existence of such structures on the 3-sphere.

\begin{corollary}\label{corsph}
The 3-sphere does not admit any conformally Anosov contact structures.
\end{corollary}

It is also worth mentioning that conformally Anosov contact structures have been previously motivated and studied in the context of the Riemannian geometry of contact manifolds as well. We discuss the two most notable instances in this paper. Firstly, Blair and Perrone \cite{blperr} proved that conformal Anosovity of a contact structure can be concluded from certain curvature conditions on \textit{compatible} Riemannian metrics. We will improve their results by showing that the very same assumptions actually imply Anosovity of such contact structures:

\begin{theorem}\label{negintro}
Let $M^{3}$ be equipped with a contact structure $\xi$ and compatible metric $g$, such that for any unit vector $e\in\xi$:

$$k(e,X_\alpha)<\left [ \frac{\theta'}{2} - \sqrt{\frac{\theta^{' 2}}{4}-\frac{1}{2}Ricci(X_\alpha)} \right ]^2 ;$$
Then $X_\alpha$ is Anosov. Here $k(e,X_\alpha)$ is the sectional curvature of the plane spanned by $e$ and the corresponding Reeb vector field $X_\alpha$, $Ricci(X_\alpha)$ is the Ricci curvature of $X_\alpha$ and $\theta'$ is the instantaneous rotation of $\xi$ with respect to $g$.
\end{theorem}

To prove the above theorem, we will use a new characterization of certain curvature quantities, derived by the author in \cite{hoz}. Furthermore, we can achieve the topological consequences of the main theorem, given the above curvature conditions. 

\begin{corollary}\label{negintrocor}
Let $(M,\xi)$ be a contact 3-manifold and $g$ a compatible Riemannian structure satisfying the upper bound on its $\alpha$-sectional curvatures given in Theorem \ref{negintro}, in particular if it has negative $\alpha$-sectional curvature. Then $(M,\xi)$ is universally tight, irreducible and does not admit any exact cobordism to $(\mathbb{S}^3,\xi_{std})$.
\end{corollary}

Moreover, Perrone \cite{perr} showed that conformal Anosovity of contact structures can be concluded assuming the existence of a compatible Riemannian metric which is the critical point of \textit{Chern-Hamilton energy functional }
$$E(g):= \int_M  |\mathcal{L}_{X_\alpha} g|^2 \hspace{0.3cm}dVol(g)$$
and is \textit{nowhere Reeb-invariant}. i.e. we have $\mathcal{L}_{X_\alpha} g \neq0$ everywhere. Here, $\alpha$ is the contact form corresponding to $g$ and $X_\alpha$ is the associated Reeb field.The significance of the study of such critical metrics is because of \textit{Chern-Hamilton conjecture} \cite{chernh} that can be generalized to:

\begin{Conjecture}
For any closed contact 3-manifold $(M,\xi)$, there exists a compatible metric that realizes the minimum (among compatible metrics) of the Chern-Hamilton energy functional.
\end{Conjecture}

The corollary of our main result is that such critical metrics on a large class of contact 3-manifolds cannot be nowhere Reeb-invariant.

\begin{corollary}\label{critintro}
Let $(M^3,\xi)$  be a contact 3-manifold which is either overtwisted, reducible or admits an exact cobordism to $(\mathbb{S}^3,\xi_{std})$ and $g$ a critical compatible metric. Then, there exists some point at which
$$\mathcal{L}_{X_\alpha}g=0,$$
where $\alpha$ is the contact form corresponding to $g$ and $X_\alpha$ is the associated Reeb field.
\end{corollary}

It is well-known (see Section~\ref{3}) that invariance of a compatible metric under the Reeb flow is equivalent to the invariance of the corresponding complex structure. The above corollary, as well as our study of Ricci curvature of compatible metrics in \cite{hoz}, leads us to the following conjecture:

\begin{Conjecture}
Let $(M^3,\xi)$ be a closed contact manifold admitting a contact form $\alpha$ and a complex structure $J$ on $\xi$ such that we have
$$\mathcal{L}_{X_\alpha} J\neq 0$$
everywhere. Then $(M,\xi)$ is universally tight.
\end{Conjecture}

In Section~\ref{2} of this paper, we briefly review the necessary backgrounds for understanding the main theorem and its proof, which includes topics from contact topology in dimension 3, Reeb dynamics, and Anosovity of vector fields. In Section~\ref{3}, although it is not necessary to understand the proof of the main theorem, we discuss motivations of our study coming from Riemannian geometry of contact manifolds and will improve a theorem of Blair-Perrone \cite{blperr}. Finally in Section~\ref{4}, we give the proof of the main theorem, as well as some of its consequences, including for the Riemannian geometric motivations.

\vspace{0.5cm}

\textbf{ACKNOWLEDGEMENT:} I greatly thank my advisor John Etnyre for his steady support, encouragement and availability to help. I am also grateful to Igor Belegradek, Sudipta Kolay, Hyun Ki Min and Andrew McCullough for helpful discussions along the way. The author was partially supported by the NSF grant DMS-1608684.

\section{Background}\label{2}

In this section, we review the necessary background for our main theorem. For more details, the interested reader is encouraged to refer to more comprehensive sources: for an introduction to contact topology in dimension 3, \cite{etnyreintro} and \cite{geiges} provide the necessary background, for the study of Anosovity in the context of contact geometry, see \cite{Mitsumatsu}, \cite{confoliations} and \cite{foulonh}, and for understanding Conley-Zehnder indices and their role in contact topology and dynamics, we refer the reader to the beautiful concise survey of \cite{hoferdy}. 

The study of Conley-Zehnder indices goes beyond contact topology. A curious reader will find \cite{Long} and \cite{gutt} insightful on this topic.

\subsection{Contact Structures in Dimension 3}

\begin{definition}
We call the 1-form $\alpha$ a contact form on $M$, if
$$\alpha \wedge d\alpha \neq 0.$$
If $\alpha \wedge d\alpha>0$ (compared to the orientation on $M$), we refer to $\alpha$ as a positive contact form and otherwise, a negative one. We call $\xi:=\ker{\alpha}$ a (positive or negative) contact structure on $M$. Moreover, we call the pair $(M,\xi)$ a contact manifold.
\end{definition}

Note that by Frobenius theorem, we can equivalently define $\xi$ as a \textit{coorientable maximally non-integrable} plane field on $M$. 

In this paper, we assume the contact structures to be positive, unless stated explicitly otherwise.

\begin{example}
Basic examples of contact structures are:

1) $\alpha_{std}=dz-ydx$ is a contact form on $\mathbb{R}^3$. We call $\xi_{std}$ the standard contact structure on $\mathbb{R}^3$.

2) Consider $\mathbb{S}^3$ as the unit sphere in $\mathbb{C}^2$. Then we can define the standard contact structure on $\mathbb{S}^3$ by $\xi_{std}:=T\mathbb{S}^3 \cap J T\mathbb{S}^3$, where $J$ is the standard complex structure on $T\mathbb{C}^2$. i.e. $\xi_{std}$ is the unique complex line tangent to the unit sphere. Moreover, we can easily show that this is the one point compactification of the standard contact structure on $\mathbb{R}^3$.

3) Let $(\Sigma,g)$ be a Riemannian surface and $\pi:UT\Sigma \rightarrow \Sigma$ its unit tangent bundle. The \textit{tautological 1-form} given by
$$\theta_u (v):=g(u,\pi_* v ) \text{  for } u \in UT\Sigma \text{ and } v \in T_u UT\Sigma$$

is a contact form on $UT\Sigma$.
\end{example}

 Darboux theorem states that all contact structures are locally \textit{contactomorphic}. i.e. locally there exists a diffeomorphism taking one contact structure to any other. In other words, contact structures do not have local invariants, in contrast to Riemannian geometry where curvature distinguishes Riemannian manifolds locally. This means that if contact geometry carries any information about the underlying manifold, it should be of global nature. Since mid 1970s, starting with the works of Bennequin, subtle relation between contact structures and 3-manifold topology has been explored and full-blown into one of the most active area of low dimensional topology. It turns out that the following dichotomy, introduced by Eliashberg \cite{elover}, plays a significant role in such study:
 
 \begin{definition}
 We call $(M,\xi)$ overtwisted if there exists an embedded disk in $M$ that is tangent to $\xi$ along its boundary. Otherwise, we call $(M,\xi)$ tight. Moreover, $\xi$ universally tight if even its lift to the universal cover of $M$ is tight.
 \end{definition}
 
 Although it is not clear from the above definition why this is a useful dichotomy, Eliashberg showed that we can reduce the study of overtwisted contact structures to the algebraic topology of the underlying manifold and classify them (in particular they always exist and are abundant). On the other hand, tight contact structures are harder to understand, classify, and they do not always exist, but when they do, are closely connected to the smooth topology of the maifold. One of the main problems in contact topology is to determine when a given contact structure is tight as well as classify such structures on a given 3-manifolds.
 
\begin{remark}
It can be shown that all the contact structures given in above examples are tight. Let us also mention that, according to to the classification of contact structures on $\mathbb{S}^3$ by Eliashberg \cite{elover}\cite{elmart}, $\mathbb{S}^3$ admits a $\mathbb{Z}$-family of distinct (up to isotopy) contact structures and $\xi_{std}$ is the only tight one.
\end{remark}

\subsection{Contact Dynamics and Conley-Zehnder Indices}

One way to study the topological properties of contact manifolds is through the dynamics of such objects. More precisely, through the dynamics of the associated \textit{Reeb vector fields}.

\begin{definition}
Let $(M,\xi)$ be a contact 3-manifold. Any choice of contact form $\alpha$ for $\xi$ defines a unique vector field $X_\alpha$ satisfying:

\center{i) $d\alpha(X_\alpha ,.)=0$}

\center{ii) $\alpha(X_\alpha)=1$ }
\end{definition}
It is easy to check:
\begin{proposition}
The Reeb vector field $X_\alpha$ satisfies:

a) $X_\alpha \pitchfork \xi$

b) $\mathcal{L}_{X_\alpha} \alpha =0$
\end{proposition}
\begin{example}
the Reeb vector field for the contact structures given above are:

1) $\partial_z$ is the Reeb vector field for $(\mathbb{R}^3,\alpha_{std})$.

2) The Reeb vector field associated to $(\mathbb{S}^3,\xi_{std})$ traces the Hopf fibration, considering an appropriate contact form for $\xi_{std}$.

3) Geodesic flow on $UT\Sigma$ is the Reeb vector field associated to $(UT\Sigma, \theta)$. It can be useful to think of Reeb vector fields as generalization of geodesic flows.
\end{example}

It is well known that these vector fields play a significant role the theory of contact geometry, comparable to the role of Hamiltonian vector fields in symplectic geometry. As a matter of fact, Reeb vector fields are Hamiltonian vector fields on the so called {\em symplectization} of a contact manifold.

However, it was not till early 1990s that dynamics of Reeb vector fields were used to study the topology of contact manifolds, thanks to many, but first and foremost, Hofer, Wysocki, and Zehnder. In order to state a particular result that we will use in the proof of the main theorem, we need to introduce an index associated to the closed orbits of Reeb vector fields.

Recall that the group of symplectic linear maps reduces to the group of area preserving linear maps in dimension 2. i.e.

$$Sp(1_{\mathbb{C}}):=\{ A \in \mathbb{R}^{2 \times 2} | A^T JA=Id\}=SL(2;\mathbb{R})$$

We can uniquely write any $A\in Sp(1_{\mathbb{C}})$ as $A=MU$, where $U\in SO(2)$ and $M$ is a symmetric positive definite matrix. Since the space of positive definite matrices is contractible, we conclude that $Sp(1_{\mathbb{C}})$ is homotopy equivalent to $SO(2)$ and therefore $\pi_1 (Sp(1_{\mathbb{C}}))=\mathbb{Z}$.

Now considering a path of symplectic maps, starting from $Id$, we want to measure its rotation around the generator of $\pi_1 (Sp(1_{\mathbb{C}}))$ (notice that we are not assuming that such path is closed). It is not necessary, but for the sake of simplicity, we restrict our definition to paths ending in the subspace of $Sp(1_{\mathbb{C}})$ given by

$$Sp^* (1_{\mathbb{C}}):=\{A \in Sp(1_{\mathbb{C}}) | det(A-Id) \neq 0 \}.$$

We call such paths of symplectic maps, \textit{non-degenerate} and denote the space of such paths by $\Sigma^*(1_\mathbb{C})$.

For non-degenerate paths of symplectic maps

$$\Phi: [0,T] \rightarrow Sp(1_{\mathbb{C}}) \text{ with }\Phi(0)=Id \text{ and } \Phi(T)\in Sp^*(1_{\mathbb{C}})$$

\noindent we can define a unique index map, defined by the following axioms, see \cite{hoferdy}:

\begin{theorem}
There exists a unique map, called \textit{Conley-Zehnder index},

$$\mu_{CZ} :\Sigma^* (1_\mathbb{C}) \rightarrow \mathbb{Z}$$
\noindent such that

1) \underline{Homotopy Invariance}: $\mu_{CZ}$ is invariant under homotopy through non-degenerate paths.

2) \underline{Maslov Compatibility}: Let $L:[0,T] \rightarrow Sp (1_{\mathbb{C}})$ be a continuous closed loop, then
$$\mu_{CZ}(L\Phi)-\mu_{CZ}(\Phi)=2 \mathcal{\mu}(L);$$
\noindent where $\mathcal{\mu}(L)$ is the \textit{Maslov index} of $L$.

3) \underline{Invertibility}:$$\mu(\Phi^{-1}) =-\mu(\Phi)$$

4) \underline{Normalization}: $$\mu_{CZ}(e^{i\pi t}|_{t\in [0,1]})=1$$

\end{theorem}

\begin{remark}\label{czbasic}
1) Notice that Maslov index assigns to any closed path in $\Sigma^*(1_\mathbb{C})$, the degree of such map. Therefore Axiom 2) means that for any full round around the generator of $\pi_1 (Sp(1_{\mathbb{C}}))$, we are adding $2$ to the Conley-Zehnder index.

2) From the above axioms, we can prove that since for any $A\in Sp(1_{\mathbb{C}})$ we have $det(A)=1$, we can determine the parity of $\mu_{CZ}(\Phi)$ by merely looking at the eigenvalues of $\Phi(T)$. More precisely, $\mu_{CZ}(\Phi)$ is even if $\Phi(T)$ has real positive eigenvalues $\lambda$ and $\frac{1}{\lambda}$ (we call $\Phi$ positively hyperbolic in this case) and $\mu_{CZ}(\Phi)$ is odd if either $\Phi(T)$ has real negative eigenvalues $\lambda$ and $\frac{1}{\lambda}$ (we call $\Phi$ negatively hyperbolic in this case) or $\Phi(T)$ has complex conjugate eigenvalues $e^{\pm i\phi}$ (we call $\Phi$ elliptic in this case).

3) There is a sophisticated iteration theory relating $\mu_{CZ}(\Phi)$ to $\mu_{CZ}(\Phi^m)$ for $m\in \mathbb{N}$, see \cite{Long} for instance. In this paper, we only deal with the easiest case which is when $\Phi$ is (positively or negatively) hyperbolic. In this case, for any $m \in \mathbb{N}$:
$$\mu_{CZ}(\Phi^m)=m\cdot \mu_{CZ}(\Phi).$$
\end{remark}

\begin{remark}\label{czdef}
1) There are other equivalent definitions of Conley-Zehnder indices. In particular, we can define $\mu_{CZ}(\Phi)$ as the algebraic intersection of $\Phi$ and the subvariety given by

$$Sp(1_{\mathbb{C}}) \setminus Sp^* (1_{\mathbb{C}}) =\{A \in Sp(1_{\mathbb{C}}) | det(A-Id) = 0 \}.$$
 We can formulate the proof of our main theorem using this definition. But we find the axiomatic definition more intuitive for a non-expert.
 
 2) We can extend the above definition to higher dimensions as well as degenerate paths of symplectic linear maps, referred to as Robbin-Salamon indices. But considering non-degenerate paths is enough for our purpose \cite{Long}\cite{gutt}.
\end{remark}

Given a periodic Reeb orbit $\gamma$ of $X_\alpha$ for a contact manifold, fixing a symplectic trivialization $\nu$ of $(\xi|_\gamma,d\alpha)$ along $
\gamma$ and picking a point $p\in \gamma$, the flow of $X_\alpha$ defines a path of symplectic linear maps

$$\Phi:[0,T]\rightarrow Sp(1_{\mathbb{C}})$$

\noindent such that $\Phi(0): \xi |_p \rightarrow \xi |_p =Id$ and $T$ is the period of $\gamma$. Notice that we used the fact that $X_\alpha$ preserves $d\alpha$ and therefore $d\alpha|_\xi$. Now we can use all the terminology of being non-degenerate, (positively or negatively) hyperbolic and elliptic, directly for the period orbit $\gamma$. Abusing notation and assuming that $det(\phi (T)-Id) \neq 0 $, we define the Conley-Zehnder index of $\gamma$ with respect to $\nu$ to be the Conley-Zehnder index of the induced path of symplectic maps and write it as $\mu_{CZ}^\nu (\gamma)$.

 Assume $[\gamma]=0\in H_1(M)$ and let $\Sigma$ be a 2-chain such that $\partial \Sigma =\gamma$. By obstruction theory, $\Sigma$ defines a trivialization of $\xi$ along $\gamma$ which is unique up to homotopy. The fact that the endpoint of the induced path is independent of such trivialization (it is simply the Poincare' return map $\Phi(T):\xi|_p \rightarrow \xi|_p$) and homotopy invariance property of Conley-Zehnder indices guarantee that Conley-Zehnder index of $\gamma$ only depends on $\Sigma$ and hence, we use the notation $\mu_{CZ}^\Sigma (\gamma)$. Moreover, it is well known (for instance see \cite{sft}) that if $\Sigma_1$ and $\Sigma_2$ are two of such 2-chains for $\gamma$, we can compute the difference of the induced Conley-Zehnder indices by

$$\mu_{CZ}^{\Sigma_1}(\gamma) - \mu_{CZ}^{\Sigma_2}(\gamma) = \langle 2e(\xi), \Sigma_1 \sqcup -\Sigma_2 \rangle$$

\noindent where $-\Sigma_2$ refers to $\Sigma_2$ with reversed orientation. Note that $\Sigma_1 \sqcup -\Sigma_2$ is a cycle and therefore, represents and element of $H_2(M)$. In particular, $\mu_{CZ}(\gamma)$ is well-defined if $2e(\xi)=0$ or $H_2(M)=0$.

\begin{remark}\label{czreeb}
 For any periodic Reeb orbit $\gamma$ and $m\in\mathbb{N}$, going $m$ rounds around $\gamma$ is also a periodic Reeb orbit which we denote by $\gamma^m$ and similar iteration theory mentioned as in Remark~\ref{czbasic} part 3) shows that fixing a symplectic trivialization of $(\xi|_{\gamma},d\alpha)$, for a (positively or negatively) hyperbolic periodic Reeb orbit $\gamma$:

$$\mu_{CZ}(\gamma^m)=m\cdot \mu_{CZ}(\gamma).$$
\end{remark}

We can finally state celebrated result of Hofer \cite{hofermain} and Hofer-Wysocki-Zehnder \cite{hwz} (\cite{wendl} can be helpful as well):

\begin{theorem}\label{hofermain}
Let $(M^{3},\xi)$ be a contact manifold which satisfies one of the followings:

1) $\xi$ is overtwisted;

2) $M$ is reducible;

3) there exists an exact symplectic cobordism from $(M,\xi)$ to $(\mathbb{S}^3,\xi_{std})$.

\noindent Then any associated Reeb vector field for $(M,\xi)$ admits a contractible unknotted periodic orbit $\gamma$ with 
$$\mu_{CZ}^D (\gamma)=2$$
in the first two case and
$$\mu_{CZ}^D (\gamma)\in \{2,3 \}$$
in the third case, where $D$ is the contraction disk for $\gamma$.
\end{theorem}

\begin{remark}
An exact symplectic cobordism from $(M_-,\xi_-)$ to $(M_+,\xi_+)$ is $(X,\omega)$ such that $\partial X=(-M_-) \sqcup (M_+)$ and there exists a global Liouville vector field $Y$ with $\alpha:=\iota_Y \omega$ being a contact form for $\xi_-$ and $\xi_+$, when restricted to the boundary of $X$. It is worth mentioning that the third case above is a very large class of contact manifolds, including all overtwisted contact manifolds (i.e. it includes case 1)) \cite{eh}.
\end{remark}
\subsection{Anosovity of Contact 3-Manifolds}

In this paper, we study contact dynamics when the Reeb vector field $X_{\alpha}$ has certain dynamical properties of Anosovity type. \textit{Conformally Anosov flows}, independently introduced by Mitumatsu \cite{Mitsumatsu} (as \textit{projectively Anosov flows}) and Elishberg-Thurston \cite{confoliations}, are generalizations of the well-studied \textit{Anosov flows}, which seem to be more natural from topological point of view. Anosov flows are known to have strong rigidity properties and put strong restrictions on the topology of the underlying manifold. Although it was not known a priori, it turns out that conformally Anosov flows are abundant \cite{asa}. The content of our main theorem is that unlike the general case, they enjoy similar rigidity properties as Anosov flows in the context of contact topology. i.e. when they show up as Reeb vector fields.

\begin{definition}
We call the vector field $X$ conformally Anosov if there exists a splitting $TM=E^s \oplus E^u \oplus \langle X \rangle$, such that the splitting is continuous and invariant under $X$ and 
$$ ||d\phi^t (v)  || / ||d\phi^t (u) || \geq Ae^{Ct}||v || /||u ||$$
for any $v \in E^u$ ({\ttfamily"}unstable direction{\ttfamily"}) and $u \in E^s$ ({\ttfamily"}stable direction{\ttfamily"}), where $\phi^t$ is the flow of $X$ and $C$ and $A$ are positive constants.
\end{definition}

Note that $E^s \oplus \langle X\rangle$ and $E^u \oplus \langle X\rangle$ are integrable, assuming $X_\alpha$ is at least $C^1$, since the splitting is invariant under the flow.

Even in the general case, we can use contact geometry to study such flows.

\begin{proposition}[Eliashberg-Thurston 1998\cite{confoliations}, Mitsumatsu 1995\cite{Mitsumatsu}]\label{caxi}
A vector field $X$ is conformally Anosov if $\langle X \rangle=\xi_{+} \cap \xi_{-}$, where $\xi_+$ and $\xi_-$ are transverse positive and negative contact structures respectively. Conversely, Any vector field directing the intersection of transverse positive and negative contact structures is conformally Anosov.
\end{proposition}
\begin{proposition}
With the above notation,
$$ E^u  = \lim_{t \rightarrow + \infty} (T\phi^t)\xi_+=\lim_{t \rightarrow + \infty} (T\phi^t)\xi_-$$

$$ E^s  = \lim_{t \rightarrow - \infty} (T\phi^t)\xi_+=\lim_{t \rightarrow - \infty} (T\phi^t)\xi_-$$
\end{proposition}

Note that the positive (negative) contact condition for $\xi_+$ ($\xi_-$) reads as $g([e_+,X] , e_-)>0$ ($g([e_-,X] , e_+)>0$), where $g$ is any arbitrary metric and $e_+$ and $e_-$ are vector fields in $\xi_+ \cap E^s \oplus E^u$ and $\xi_- \cap E^s \oplus E^u$ respectively, such that $(e_+,e_-)$ is an oriented basis for $\xi$. This means that the image of $e_+$ ($e_-$) under the flow {\ttfamily"}rotates{\ttfamily"} in a right (left) handed fashion, from $E^s$ towards $E^u$ (see Figure 1(a)).

 \begin{figure}

 \begin{subfigure}[b]{0.4\textwidth}

  \center \begin{overpic}[width=4cm]{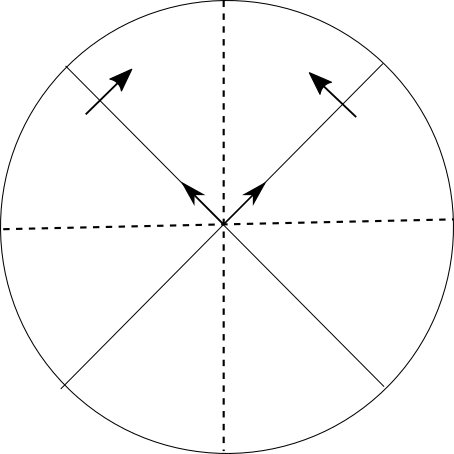}
  \put(72,65){$e_+$}
  \put(32,65){$e_-$}
         \put(10,105){$\xi_-$}
         \put(95,105){$\xi_+$}
         \put(57,120){$E^u$}
         \put(120,60){$E^s$}
       \end{overpic}
    \caption{Confomally Anosov dynamics}
    \label{fig:1}
  \end{subfigure}
  \hspace{2cm}
  \begin{subfigure}[b]{0.4\textwidth}
  \center \begin{overpic}[width=7cm]{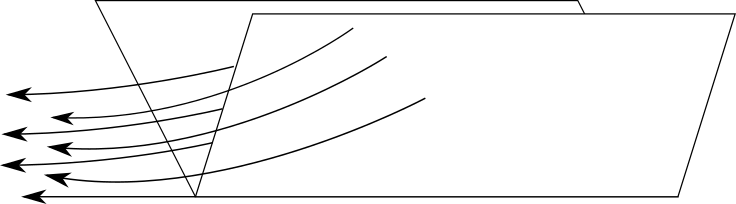}
  \put(-20,10){$X$}
  \put(30,48){$\xi_-$}
  \put(110,30){$\xi_+$}
  \end{overpic}
    \vspace{0.3cm}
    \caption{Conformally Anosov flows as intersection of positive and negative contact structures}
    \label{fig:2}
  \end{subfigure}
  
 \caption{Conformally Anosov flows}
\end{figure}
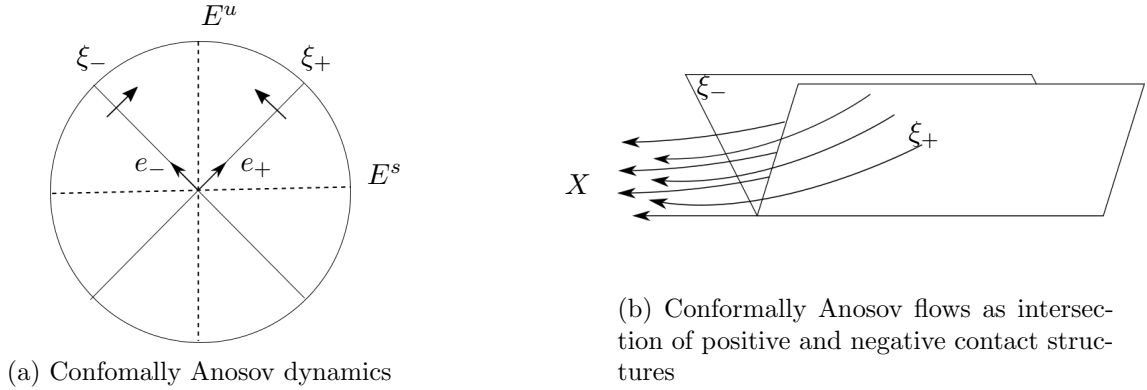

An Anosov vector field $X$ can similarly be defined as a vector field which preserves a continuous splitting $TM=E^s \oplus E^u \oplus \langle X \rangle$ and
$$||d\phi^t (v)  ||  \geq A_1 e^{C_1t}||v || \text{  for any } v \in E^u$$
$$||d\phi^t (u) || \leq A_2 e^{-C_2t}||u || \text{  for any } u \in E^s,$$
\noindent where $\phi^t$ is the flow of $X$ and $C_1$, $C_2$, $A_1$ and $A_2$ are some positive constants. Prototypical examples of Anosov flows are geodesic flows on $UT\Sigma$, where $\Sigma$ is a hyperbolic surface and suspensions of Anosov diffeomorphisms of $\mathbb{T}^2$. Anosov flows enjoy various topological properties. In particular, the underlying manifold needs to have a fundamental group with \textit{exponential growth} (for classic results about Anosov flows refer to \cite{anosov}\cite{plante})This rules out the possibility of manifolds like $\mathbb{S}^3$ or $\mathbb{T}^3$ admitting such flows. Conformally Anosov flows can be defined as flows which have the same action on the projectified tangent space as of Anosov flows. By now, it is known that these flows are abundant. For instance, both $\mathbb{S}^3$ and $\mathbb{T}^3$ admit infinitely many (distinct up to isotopy) conformally Anosov flows \cite{asa}.

We are interested in these flows in the category of contact manifolds.

\begin{definition}
We call a conformally Anosov vector field $X$ contact if $E^s \oplus E^u =\ker{\alpha}$ for some smooth contact form $\alpha$. Equivalently, we call $\xi:=\ker{\alpha}$ a conformally Anosov contact structure.
\end{definition}
 Notice that $X$ is a Reeb vector field for the contact structure $E^s \oplus E^u$ for an appropriate choice of contact form, since $X$ preserves the plane $E^s \oplus E^u$. Also it is well known that smoothness of the corresponding contact form yields the smoothness of stable and unstable foliations.

 Similarly, we can define contact Anosov flows and Anosov contact structures (when $X$ is Anosov). The kernel of the tautological 1-form on $UT\Sigma$, $\Sigma$ being a hyperbolic surface, is an example of Anosov contact structures. In fact, these have been the only known examples, until recently when \cite{foulonh} gave the first examples of Anosov contact structures on a family of hyperbolic 3-manifolds. It is not known much how bigger the class of conformally Anosov contact structures are, but our main result will show that they do satisfy similar contact topological properties of Anosov contact structures.

 \section{Riemannian Geometry of Conformally Anosov Contact 3-Manifolds}\label{3}
 
 Conformally Anosov contact 3-manifolds have been previously studied, mainly by Blair and Perrone, in the context of Riemannian geometry of contact manifolds. In this section, we will review previous results along these lines in order to state the consequences of our main theorem for \textit{{\ttfamily"}compatible Riemannian geometry{\ttfamily"}}, as well as to improve one of the results. The first result by Blaire and Perrone \cite{blperr} establishes conformal Anosovity of contact structures assuming certain restriction on curvature of a \text{compatible} Riemannian metric. We will improve this result to Anosovity of the underlying contact structure. In order to do so, we will use a new characterization of certain sectional and Ricci curvatures, given in \cite{hoz} by the author, which we find more suitable for our purpose. The second result, due to Perrone \cite{perr}, concludes conformal Anosovity based on the existence of a \textit{nowhere Reeb-invariant critical metric}. Riemannian geometry of contact manifold has been classically studied widely and we do not intend to draw a comprehensive picture of the topic. A classic reference in this field of study is \cite{bl}, which the curious reader should consult. The necessary background, in a spirit closer to our viewpoint, is also provided in \cite{etkmass1} and \cite{etkmass2}.
 
\subsection{Compatibility}

Given a contact 3-manifold $(M,\xi)$, we can naturally focus on a class of Riemannian metrics that are \textit{compatible} with $\xi$.

\begin{definition}
A Riemannian structure $g$ is compatible with $(M,\xi)$ if 
$$ g(u,v)=\frac{1}{\theta '} d\alpha(u,Jv)+\alpha(u)\alpha(v)$$
for $u,v \in TM$, where $\alpha$ is a contact form for $\xi$, $\theta '$ is a positive constant (called \textit{{\ttfamily"}instantaneous rotation{\ttfamily"}}) and $J$ is a complex structure on $\xi$, naturally extended to $TM$ by first projecting along the Reeb vector field associated with $\alpha$.

\end{definition}
 
 \begin{example}
$(\mathbb{S}^3,\xi_{std})$ and $(UT\Sigma, \theta)$ are compatible with round metric and the natural metric induced $UT\Sigma$ from the assumed metric on $\Sigma$, respectively.
 \end{example}
 
 \begin{remark}
  1) The positive real number $\theta '$ measures how fast the contact planes are {\ttfamily"}rotating{\ttfamily"}. Let $\{ u,v \}$ be an orthonormal basis for a plane field $\xi$ in a neighborhood and $n$ its normal vector field. At any point we can define:
 $$  \theta (t):= \cos^{-1} \left( \frac{g((\phi_{-t})_* v,n)}{||\phi_{-t})_* v||} \right)$$
 where $\phi$ is the flow of $u$. Note that $\theta ':=\theta' (0)=-g([u,v],n)=d\alpha (u,v)$ where $\alpha$ is the dual of $n$. The inequality $\theta ' >0$ is equivalent to the contact condition and for the compatibility definition, we consider this number to be constant on $M$.

2) The Reeb vector field $X_\alpha$ is orthonormal to $\xi$ and moreover, is a geodesic field. This helps us use \textit{Jacobi fields} as an effective tool in understanding the geometry and dynamics of such structures.

3) In the classical literature like \cite{bl}, the case of $\theta' =2$ is studied (named as \textit{{\ttfamily"}contact metrics{\ttfamily"}}), while we do not see such restriction necessary for our work.
 
 \end{remark}
 
Now for any plane field $\xi$ on a Riemannian manifold $(M,g)$, we can define an \textit{second fundamental form} by:

$$\mathbb{II}(u,v)= g(\nabla_u v,n)$$

\noindent where $n$ is the unit normal to $\xi$ and $u,v \in \xi$.

It can be easily shown that $\mathbb{II}$ is symmetric if and only if $\xi$ is integrable. We can define two geometric invariants of $\xi$ using this second fundamental form, namely the \textit{mean curvature} $ H(\xi):=trace (\mathbb{II}) $ and the \textit{extrinsic curvature} $ G(\xi):= det (\mathbb{II}(\xi))$. 

It can be seen that if $\xi$ is a contact structure and $g$ a compatible Riemannian metric, we will have

 $$ H(\xi)= - div_g (X_\alpha)=0,$$
 
 \noindent while we show in \cite{hoz} that $G(\xi)$ can be interpreted as the \textit{Ricci curvature} of $X_\alpha$:

 \begin{theorem}\label{cur}
 Let $(M^{3},\xi)$ be a contact 3-manifold and $g$ a compatible Riemannian metric. Then for any unit vector $e\in\xi$, the sectional curvature of the plane $\langle e,X_\alpha \rangle$ can be computed as:
 
 $$k(e,X_\alpha)=g(Je,\nabla_e X_\alpha)^2-g(e,\nabla_e X_\alpha)^2-\frac{\partial}{\partial t}g(e(t),\nabla_{e(t)} X_\alpha)|_{t=0}$$
 
 \noindent where $e(t):=\frac{\tilde{e}(t)}{|\tilde{e}(t)|}$ and $\tilde{e}(t)$ is the unique (locally defined) Jacobi field along $X_\alpha$ with $\tilde{e}(0)=e$ and $D\tilde{e}(0)=\nabla_{\tilde{e}(0)}X_\alpha$. Moreover,
 $$Ricci(X_\alpha):= k(e,X_\alpha)+k(Je,X_\alpha)=2\cdot G(\xi)=\frac{\theta'^2}{2}-2\cdot g(e,\nabla_e X_\alpha)^2- 2\cdot \left ( g(Je,\nabla_e X_\alpha)-\frac{\theta'}{2} \right )^2.$$
  \end{theorem}

\begin{remark}\label{nav}
Using the above characterization, we see:
$$Ricci(X_\alpha)\leq \frac{\theta'^2}{2},$$
and it can be easily shown \cite{hoz} that the equality holds at a point if and only if $\mathcal{L}_{X_\alpha}J=0$ (or equivalently $\mathcal{L}_{X_\alpha}g=0$), if and only if $g(e,\nabla_e X_\alpha)=0$ for \textit{every} unit vector $e\in \xi$ at that point. Noting that $H(\xi)=g(e,\nabla_e X_\alpha)+g(Je,\nabla_{Je} X_\alpha)=0$, that also means that if we have $\mathcal{L}_{X_\alpha}g \neq 0$ at a certain point, we get a natural splitting of $\xi$ at that point. More precisely, there will be exactly two orthogonal directions $\langle e_1 \rangle$ and $\langle e_2 \rangle$ for which $g(e_i,\nabla_{e_i} X_\alpha)=0$ ($i=1,2$), partitioning $\xi$ into four quadrants with alternating signs for $g(e,\nabla_e X_\alpha)\neq 0$ (see Figure 2). If such condition holds globally, this would be equivalent to defining a sub line field of $\xi$ which is useful for topological purposes. Moreover, the term $\frac{\partial}{\partial t}g(e(t),\nabla_{e(t)} X_\alpha)|_{t=0}$ in $k(e,X_\alpha)$ helps us understand the interplay of local behavior of the flow of $X_\alpha$ and curvature. Notice that since $X_\alpha$ is a geodesic field, we can use Jacobi fields associated to the variations of the Reeb vector field, which in \cite{hoz} we refer to them as $\alpha$-Jacobi fields, to study the geometry and dynamics of the Reeb flow. These are exactly the Jacobi fields $\tilde{e} (t)$, defined on a segment of Reeb flow and satisfying $[X_\alpha , \tilde{e} (t)]=0$ or equivalently $D\tilde{e}=\nabla_{\tilde{e}}X_\alpha$. We will exploit this viewpoint in the proof of Theorem \ref{neg}.
\end{remark}

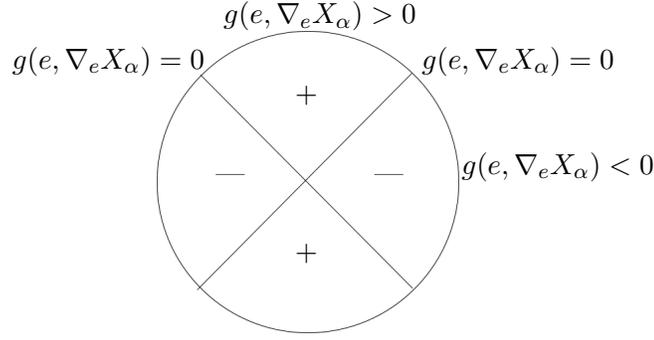
\begin{figure}
  \center \begin{overpic}[width=4cm]{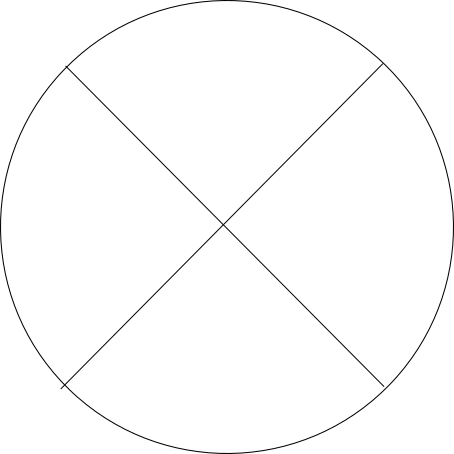}
  \put(-55,100){$g(e,\nabla_e X_\alpha)=0$}
  \put(100,100){$g(e,\nabla_e X_\alpha)=0$}
    \put(25,117){$g(e,\nabla_e X_\alpha) >0$}
      \put(115,60){$g(e,\nabla_e X_\alpha)<0$}
  \put(52,87){+}
  \put(22,57){\textemdash}
    \put(52,27){+}
  \put(82,57){\textemdash}
  
  \end{overpic}
\caption{Splitting of $\xi$ when $\mathcal{L}_{X_\alpha}g\neq 0$ and regions with alternating signs of $g(e,\nabla_e X_\alpha)$}
\end{figure}

\begin{remark}
In the classical literature \cite{rei}, the second fundamental form of a plane field is usually referred to the symmetric form derived from the above bilinear form. We do not consider this symmetrization and in fact, the second fundamental form defined above is asymmetric for contact structures.
\end{remark}

\subsection{Compatible Metrics with Negative $\alpha$-Sectional Curvature}

In \cite{blperr}, Blair and Perrone proved that for a contact structure and a compatible metric, if the sectional curvature of the planes including the Reeb direction satisfies a certain upper bound, in particular if it is negative, we can conclude the conformal Anosovity of the underlying contact structure. Calling such sectional curvatures $\alpha$-sectional curvatures, we will reprove their result using the characterization given above, as well as improve it to show that such contact structures are in fact Anosov. It is worth mentioning that the compatible Riemannian structure on $(UT\Sigma,\theta)$, when $\Sigma$ is a hyperbolic surface, satisfies this condition.

\begin{theorem}\label{neg}
Let $M^{3}$ be equipped with a contact structure $\xi$ and a compatible metric $g$, such that for any unit vector $e\in\xi$:

$$k(e,X_\alpha)<\left [ \frac{\theta'}{2} - \sqrt{\frac{\theta^{' 2}}{4}-\frac{1}{2}Ricci(X_\alpha)} \right ]^2$$

Then $X_\alpha$ is Anosov.
\end{theorem}

\begin{proof}
Let $e_1, e_2 \in \xi$ be non-parallel unit vectors with $g(e_i, \nabla_{e_i} X_\alpha)=0$ for $i=1,2$ (see Remark~\ref{nav}). Then by Theorem \ref{cur}, we can easily compute for $i=1,2$:
$$g(Je_i,\nabla_{e_i} X_\alpha )^2=\left [ \frac{\theta'}{2} \pm \sqrt{\frac{\theta^{' 2}}{4}-\frac{1}{2}Ricci(X_\alpha)} \right ]^2$$

and our assumption on sectional curvature will imply:
$$\frac{\partial}{\partial t} g(e_i(t),\nabla_{e_i(t)} X_\alpha)>0.$$

Now by Proposition~\ref{caxi} and the following discussion, we know that $X_\alpha$ is conformally Anosov, since $\langle e_1,X_\alpha \rangle$ and $\langle e_2,X_\alpha \rangle$ are positive and negative contact structures. Also notice that this implies that for any $e\in E^u$ ($e\in E^s$) of this conformally Anosov flow, we have $g(e,\nabla_e X_\alpha) >0$ ($g(e,\nabla_e X_\alpha) <0$). See Figure 3.

Now in order to to prove that $X_\alpha$ is furthermore Anosov, choose $C>0$ such that for any unit vector $e\in E^u$ at any point in $M$, $g(e,\nabla_e X_\alpha) >C$ holds (such $C$ exists by the compactness of $M$). Using the notation of Theorem \ref{cur}, we will have:

$$g(e(t),\nabla_{e(t)} X_\alpha) =\frac{1}{2}\frac{\partial}{\partial t} \ln{g(\tilde{e}(t),\tilde{e}(t))}>C$$
and this implies
$$ \ln{g(\tilde{e}(t),\tilde{e}(t))} - \ln{g(\tilde{e}(0),\tilde{e}(0))}>Ct$$

$$ g(\tilde{e}(t),\nabla_{\tilde{e}(t)} X_\alpha) > e^{Ct}.$$

A similar argument for the unit vector $e \in E^s$ yields the Anosovity of $X_\alpha$.
\end{proof}

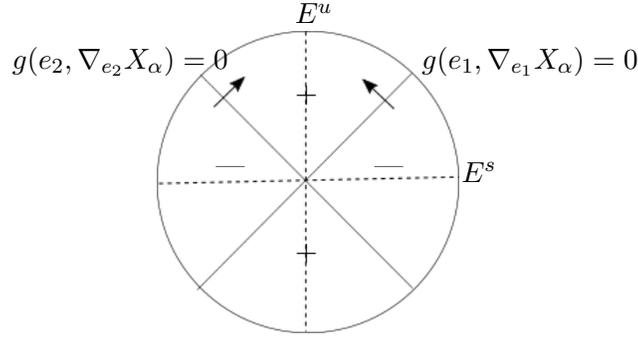
\begin{figure}
  \center \begin{overpic}[width=4cm]{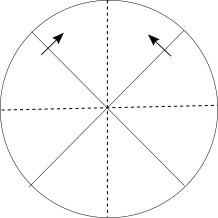}
  \put(-55,100){$g(e_2,\nabla_{e_2} X_\alpha)=0$}
  \put(100,100){$g(e_1,\nabla_{e_1} X_\alpha)=0$}
  \put(115,57){$E^s$}
  \put(52,117){$E^u$}
  \put(52,87){+}
  \put(22,60){\textemdash}
    \put(52,27){+}
  \put(82,60){\textemdash}
  
  \end{overpic}
\caption{Dynamics of contact structures admitting a compatible metric with negative $\alpha$-sectional curvature}
\end{figure}

\subsection{Nowhere Reeb-Invariant Critical Metrics}

In \cite{chernh}, Chern and Hamilton initiated the study of a particular class of compatible metrics, namely critical compatible metric, by stating a conjecture that can be generalized to:

\begin{Conjecture}
For any closed contact 3-manifold $(M,\xi)$, there exists a compatible metric that realizes the minimum (among compatible metrics) of the energy functional:

$$E(g):= \int_M  |\mathcal{L}_{X_\alpha} g|^2 \hspace{0.3cm}dVol(g).$$
\end{Conjecture}

Motivated by this conjecture, we can study the critical points of this energy functional restricted to the space of compatible metrics. We call such metrics \textit{critical compatible metrics}. 

This conjecture was proved by Rukimbira \cite{ruk} for a very specific class of contact manifolds, namely \textit{the generalized Boothby-Wang fibrations}, by characterizing such contact manifolds as the ones admitting a compatible metric with $$\mathcal{L}_{X_\alpha} g =0$$

\noindent everywhere and therefore satisfying the condition of Chern-Hamilton conjecture.

However, Perrone \cite{perr} showed that under the extreme opposite assumption of the compatible metric being \textit{nowhere Reeb-invariant}, i.e. assuming $$\mathcal{L}_{X_\alpha} g \neq 0$$

\noindent everywhere, the existence of such critical compatible metric will imply the conformal Anosovity of the underlying contact structure.

\begin{theorem}\cite{perr}\label{crit}
If $g$ is a compatible metric which is the critical point of $E$ and $\mathcal{L}_{X_\alpha} g \neq 0$ everywhere, then $X_\alpha$ is conformally Anosov with respect to such metric.
\end{theorem}

 Our main theorem will use this fact to show that for a wide range of contact manifolds, critical metrics cannot be nowhere Reeb-invariant.
 
 \begin{remark}
Perrone \cite{perr} refers to compatible metrics with $\mathcal{L}_{X_\alpha} g \neq 0$ as {\ttfamily"}\textit{non-Sasakian} metrics{\ttfamily"}.
 \end{remark}
\section{Main Theorem}\label{4}

Anosov contact structures enjoy topological rigidity properties like tightness, thanks to the fact Anosov flows do not admit contractible periodic Reeb orbits. We show that the same properties hold for conformally Anosov contact structures, although their Reeb vector fields might a priori have contractible periodic Reeb orbits. We establish these facts using the computation of certain Conley-Zehnder indices and previous works of Hofer \cite{hofermain} and Hofer-Wysocki-Zehnder \cite{hwz} on contact dynamics. 

\begin{theorem}\label{main}
Let $(M,\xi)$ be a conformally Anosov contact manifold and $X_\alpha$ some associated Reeb vector field for $\xi$. Then:

1) $2e(\xi)=0 \in H^2 (M);$

2) All the periodic Reeb orbits of $X_\alpha$ are non-degenerate and (negatively or positively) hyperbolic;

3) For any periodic Reeb orbit $\gamma$ with $[\gamma]=0 \in H_1(M)$, we have $\mu_{CZ} (\gamma)=0$;

4) $(M,\xi)$ is universally tight, irreducible and it does not admit any exact cobordism to $(\mathbb{S}^3 , \xi_{std})$.
\end{theorem}

\vspace{0.5cm}

\begin{proof}
1) By Kobayashi \cite{Kobayashi} (see Theorem 2 and 5), $2e(\xi)=0 \in H^2 (M)$ if and only if $\xi$ admits a line sub bundle. In this case, consider the line sub bundle $E^u$.

\vspace{0.5cm}

2) The dynamics of $\xi = E^s \oplus E^u$ forces the Poincare return map along $\gamma$ to have to distinct eigenspaces with real eigenvalues $\lambda_1 , \lambda_2$ (since the stable and unstable directions are preserved). Also we see that $|\lambda _i|\neq 1$, since in that case either $\lambda_1=\lambda_2=1$ or $\lambda_1=\lambda_2=-1$, which corresponds to the Poincare return map being $Id$ or $-Id$. This is in contradiction with conformal Anosov dynamics on $\xi = E^s \oplus E^u$.

\vspace{0.5cm}

3) The rough idea for computing Conley-Zehnder indices is that in this case, the Reeb flow does not {\ttfamily"}rotate{\ttfamily"} with respect to the splitting $\xi = E^s \oplus E^u$. But this splitting does not necessarily induce a trivialization of the contact structure, restricted to the periodic Reeb orbit, since the stable (or unstable) line fields are not necessarily orientable.

In order to compute the Conley-Zehnder index for $\gamma$, we need a symplectic trivialization of $\xi |_\gamma$. Since $[\gamma]=0 \in H^2(M)$, we can find a Seifert surface $\Sigma_1 \subset M$ for $\gamma$ (in particular $\partial \Sigma_1 =\gamma$). The splitting $E^s \oplus E^u$ on $\Sigma_1$ would induce a trivialization of $\xi|_\gamma$, if $E^u |_{\Sigma_1}$ and $E^s |_{\Sigma_1}$ are orientable. Note that orientability of one will imply orientability of the other one, since $\xi$ is coorientable. However this is not the case in general. Also note that we need to choose the vectors of the {\ttfamily"}right{\ttfamily"} length in the directions of these line bundles to make the trivialization symplectic.

Now let $\pi : \Sigma _2 \rightarrow \Sigma_1$ to be the orientation double cover for $E^u |_{\Sigma_1}$ and note that $\pi |_{\partial \Sigma_2}$ is a double covering map for $\gamma= \partial \Sigma_1$. This induces a trivialization on $\gamma^2$, since the lift of $E^u$ to $\Sigma_2$ is orientable.

Since the splitting is preserved by the Reeb flow, the induced path of symplectic maps only includes the ones with positive real eigenvalues and therefore, we have $\mu^{\Sigma_2}_{CZ}(\gamma^2)=0$ (see Remark~\ref{czbasic}).

But since $2e(\xi)=0$, the Conley-Zehnder index is the same for any other choice of trivialization induced from a 2-chain, in particular from $2\Sigma_1$. i.e.

$$\mu_{CZ} ^{2\Sigma _1}(\gamma^2) = \mu_{CZ} ^{\Sigma_2} (\gamma^2)=0.$$

Now noting that $\gamma$ is a hyperbolic periodic Reeb orbit, by Remark \ref{czreeb}:

$$\mu_{CZ}^{2\Sigma_1} (\gamma^2)=2\cdot \mu_{CZ}^{\Sigma_1} (\gamma)=0$$ and thus $$\mu_{CZ}(\gamma)=0.$$

\vspace{0.5cm}

4) The implications follow from the Theorem \ref{hofermain}, since in part 3) we showed that contractible periodic Reeb orbits in this case need to have Conley-Zehnder index equal to 0.

\end{proof}

\begin{remark}
 By the above argument, even if $[\gamma]\neq 0 \in H_1(M)$, the parity of $\mu_{CZ} (\gamma)$ (which is independent of the choice of trivialization) is determined by the orientability of $E^s$ (or $E^u$) on $\gamma$. That is $\mu_{CZ} (\gamma)$ is even (and actually zero with respect to such trivialization) if $E^s |_\gamma$ is orientable and is odd otherwise.

\end{remark}

Now can  we observe some consequences of our main theorem. First we notice that by Eliashberg's classification of contact structures on $\mathbb{S}^3$ \cite{elover}\cite{elmart}, all contact structures are included in the above theorem. Therefore, we will see that Corollary~\ref{corsph} is true. That is, $\mathbb{S}^3$ does not admit any conformally Anosov contact structures.

We can also easily conclude the Riemannian geometric consequences, mentioned in Corollary~\ref{negintrocor} and~\ref{critintro}, by applying Theorem~\ref{neg} and \ref{crit}.


\end{document}